\documentclass[final]{siamltex}

\usepackage{color}
\usepackage{a4,amsmath,amssymb,amscd,latexsym} 
\usepackage{bbm} 
\usepackage{epsfig,color} % sollte vor german stehen

\newtheorem{example}{Example}

\def\norm#1{\|#1\|}

\def\hess#1{\mbox{\sl Hess}_{\! #1}\, }

\def\cN{{\cal N}}
\def\cP{{\cal P}}
\def\exp{\mathrm{exp}}
\def\End{\mathrm{End}}
\newcommand{\N}{{\mathbbm{N}}} 
\newcommand{\R}{{\mathbbm{R}}} 

\def\scp#1{\left\langle #1\right\rangle}
\def\norm#1{\left\| #1\right\|}

\newenvironment{revision} {\color{black}} {\color{black}}
\newenvironment{rrev} {\color{black}} {\color{black}}

\title{A Riemannian View on Shape Optimization
}

\author{Volker H.~Schulz\thanks{University of Trier, Department of Mathematics, 54296 Trier, Germany ({\tt volker.schulz@uni-trier.de}).}
        }

\begin{document}

\maketitle

\begin{abstract}
Shape optimization based on the shape calculus is numerically mostly performed by means of steepest descent methods. This paper provides a novel framework to analyze shape-Newton optimization methods by exploiting a Riemannian perspective. A Riemannian shape Hessian is defined \begin{revision} possessing\end{revision}  often sought properties like symmetry and quadratic convergence for Newton optimization methods.
\end{abstract}



\pagestyle{myheadings}
\thispagestyle{plain}

\section{Introduction}
Shape optimization is a vivid field of research. In particular, the usage of shape calculus for practical applications has increased steadily \cite{Arian-1995,ARTASAN-1996,Arian-Vatsa-1998,ESSI-2009,SS-2009,CoCy2010,ComFluid2011,GISS2012}. Standard references for shape optimization based on the shape calculus are \cite{Delfour-Zolesio-2001,Mohammadi-2001,Sokolowski-1992}. The shape Hessian is already used as a means to accelerate gradient based shape optimization methods \cite{Epp-Har-2005,ESSI-2009,GISS2012,Schmidt-thesis}. It is also used for characterizing the well-posedness and sufficient optimality conditions \cite{Eppler-2000,EH-2012} in particular applications and is reformulated in the framework of differential forms in \cite{Hipt-2011}. A fairly general framework for descent methods for shape optimization is presented in \cite{Hintermueller-2005}. \begin{revision} Furthermore, general optimization strategies in shape space are discussed in \cite{Ring-Wirth-2012}.\end{revision}

However, a general framework for the analysis of Newton-type shape optimization algorithms is still missing. A major reason is the lack of symmetry \cite{Delfour-Zolesio-2001,Hipt-2011} of the shape Hessian, as it is commonly defined. \begin{revision}In \cite{HR-2003, HR-2004}, the lack of symmetry is circumvented by the choice of certain perturbation fields\end{revision}. In this paper, an attempt is made to cast shape optimization problems in the framework of optimization on Riemannian manifolds. There is a fairly large amount of publications available on the issue of optimization on Riemannian manifolds---mainly for matrix manifolds as in \cite{Absil-book-2008,Qi-2011}. 

It is proposed in this paper, to view the set of all shapes as a Riemannian manifold and follow there the ground breaking work in \cite{MM-2006,Michor06anoverview,MM-2005,BHM-2011,BHM-2011-unpublished}. The resulting manifold is an abstract infinite dimensional manifold, in contrast to the finite dimensional submanifolds of $\R^n$ that arise in optimization on matrix manifolds. Therefore, distance concepts have to be reviewed and used with somewhat more care. The key observation of this paper is that the action of an element of the tangent space of the manifolds of shapes can be interpreted as the shape derivative of classical shape calculus. Once this link is established, the concept of Riemannian shape Hessian, shape Taylor series, shape Newton convergence and sufficient shape optimality conditions follow quite naturally. 

In section \ref{sec2} the notation for the manifold of shapes is introduced and for a particular example of a Riemannian geometry, the correspondence between Riemannian geometry and shape calculus is established. The key element of the covariant derivative is rephrased in terms  of the shape calculus. The Riemannian shape Hessian is defined and the Riemannian shape Taylor series formulated. Section \ref{sec:Newton} presents a generalization of the Newton convergence theory, established in \cite{BockHxG1987a} for linear spaces, on Riemannian manifolds. From that, convergence properties of variants of Newton's method on Riemanian manifolds follow immediately. Finally, section \ref{numex} discusses numerical experiments for shape optimization algorithms with linear and quadratic convergence properties.

\section{Riemannian Shape Geometry and the Shape Calculus}\label{sec2}
The purpose of this section is to demonstrate the possibility to define a Riemannian metric on the manifold of all possible shapes with a relation to shape calculus. Since this point of view is new, we use the established framework of differential geometry for shapes which are $C^\infty$ embeddings of the unit sphere. Of course, this framework has to be generalized for specific applications. However, the purpose of this paper is to convey a new point of view rather than the impact on applications. Therefore, we assume in the interest of simplicity of the presentation maximum smoothness and restrict ourself to only 2D problems.

It is assumed that the reader is familiar with the basic concepts of Riemannian geometry as they are, e.g., geodesics, exponential mapping, parallel transport. In \cite{MM-2006,Michor06anoverview}, a geometric structure of two-dimensional $C^\infty$-shapes has been introduced and consequently generalized to shapes in higher dimension in \cite{MM-2005,BHM-2011,BHM-2011-unpublished}. Essentially, closed curves (and closed higher dimensional surfaces) are identified with mappings of the unit sphere to any shape under consideration. In two dimensions, this can be naturally motivated by the Riemannian mapping theorem. In this paper, we focus on two-dimensional shapes as subsets of $\R^2$ for ease of discussion, but mention the other publications above, in order to indicate that natural extensions of this paper to higher dimensional surfaces are conceivable. 

Here, we mean with "shape" a \begin{revision}simply\end{revision} connected and compact subset $\Omega$ of $\R^2$ with $\Omega\neq\emptyset$ and $C^\infty$ boundary $\partial\Omega$. As always in shape optimization, the boundary of the shape is all that matters. Thus, we can identify the set of all shapes with the set of all those boundaries. In \cite{MM-2006}, this set is characterized by  
\[
B_e(S^1,\R^2):=\mbox{Emb}(S^1,\R^2)/\mbox{Diff}(S^1)
\] 
i.e., as the set of all equivalence classes of $C^\infty$ embeddings of $S^1$ \begin{revision}into\end{revision} the plane ($\mbox{Emb}(S^1,\R^2)$), where the equivalence relation is defined by the set of all $C^\infty$ re-parameterizations, i.e., diffeomorphisms of $S^1$ into itself ($\mbox{Diff}(S^1)$). The set $B_e$ is considered as a manifold in \cite{MM-2006} and various Riemannian metrics are investigated. A particular point on the manifold $B_e(S^1,\R^2)$ is represented by a curve $c:S^1\ni\theta\mapsto c(\theta)\in\R^2$. Because of the equivalence relation ($\mbox{Diff}(S^1)$), the tangent space is isomorphic to the set of all normal $C^\infty$ vector fields along $c$, i.e.
\[
T_cB_e\cong\{h\ |\ h=\alpha\vec{n},\, \alpha\in C^\infty(S^1,\R)\}
\]
where $\vec{n}$ is the unit exterior normal field of the shape $\Omega$ defined by the boundary $\partial\Omega=c$ such that $\vec{n}(\theta)\perp \begin{revision}c^\prime\end{revision}$ for all $\theta\in S^1$ \begin{revision}and $c^\prime$ denotes the circumferential derivative as in \cite{MM-2006}\end{revision}. For our discussion, we pick among the other metrics discussed in \cite{MM-2006} the metric family for
$A\ge 0$
\begin{align*}
G^A:\ &T_cB_e\times T_cB_e\to \R\\
    & (h,k)\mapsto \int_{S^1}(1+A\kappa_c(\theta)^2)\scp{h(\theta),k(\theta)}\norm{c_{\theta}(\theta)}d\theta
\end{align*}
where $\kappa_c$ denotes the curvature of the curve $c$ and \begin{revision}$\scp{x,y}:=x_1y_1+x_2y_2$\end{revision} and \begin{revision}$\norm{x}:=\sqrt{\scp{x,y}}$\end{revision} mean the standard Euclidian scalar product and norm in $\R^2$. \begin{revision} Besides this metric,\end{revision}
there are many more Riemannian metrics $G$ available which can be used in a similar way. \begin{revision}For instance,
in \cite{SMSY11} a Sobolev-type metric is adapted to the particular needs of image tracking. \end{revision} If $h=\alpha\vec{n}$ and $k=\beta\vec{n}$, then this scalar product on $T_cB_e$ can be expressed more \begin{revision}simply\end{revision} as
\[
G^A(h,k)=\int_{\partial\Omega}(1+A\kappa_c^2)\alpha\beta ds
\]
where $ds$ is the length measure on $\partial\Omega=c$. In \cite{MM-2006} it is shown that for $A>0$ the scalar product $G^A$ defines a Riemannian metric on $B_e$ and 
thus, geodesics can be used to measure distances, where $d_{G^A}(.,.)$ denotes the corresponding geodesic distance. 
Unfortunately, this is not the case for the most simple member $G^0$ of the metric family $G^A$, where $A=0$. An illustrative counter-example is given in \cite{MM-2006}. In section \ref{sec:Newton}, we use the fact that $(B_e, G^A)$ is a Riemannian manifold extensively and exploit the existence of the exponential map ($\exp$) according to the usual definition in Riemannian geometry.

Now, we want to analyze the correlation of the Riemannian geometry on $B_e$ with shape optimization. In 2D shape optimization, one searches for the solution of optimization problems of the form
\[
\min\limits_{\Omega}f(\Omega)
\]
where $f$ is a real valued shape differentiable objective function. Often, the problem formulation involves explicit constraints in the form of differential equations and additional state variables as in like in the early publications \cite{Arian-1995,Delfour-Zolesio-2001,Mohammadi-2001,Sokolowski-1992}. Also shape Hessian preconditioning studies have been performed in \cite{ESSI-2009,SS-2009,Epp-Har-2005,EH-2012}. For the sake of ease of presentation, we can assume all those possible additional structures are contained implicitly within the mapping $f$. The shape derivative of $f$ is a directional derivative in the direction of a $C^\infty$ vector field $V:\R^2\to\R^2$ which can be represented on the boundary according to the Hadamard structure theorem \cite{Sokolowski-1992} as \begin{revision}a scalar distribution $g$ on the boundary. If $g\in L^1(\partial\Omega)$, the shape derivative can be expressed as 
\[
df(\Omega)[V]=\int_{\partial\Omega}g\, \scp{V,\vec{n}}\, ds\, .
\]
\end{revision}
If $V|_{\partial\Omega}=\alpha\vec{n}$, this can be written more concisely as
\[
df(\Omega)[V]=\int_{\partial\Omega}g\, \alpha\, ds
\]
In Riemannian geometry, tangential vectors are considered as directional derivatives of scalar valued functions. Since curves $c\in B_e$ can be interpreted as boundaries of domains $\Omega$ with boundary $c=\partial\Omega$, we can consider every scalar valued function $f:c=\partial\Omega\mapsto \R$ also as mapping $f:\Omega\mapsto\R$. 
Thus, we see that the action of a tangent vector $h\in T_cB_e$ on a scalar valued function $f:B_e\to\R$ can be interpreted in the shape calculus, via the unique identification of the boundary $c=\partial\Omega$  with its shape $\Omega$, as the shape derivative of $f$ with respect to an arbitrary $C^\infty$ extension $V$ of $h$ in the whole domain $\Omega$ with $V|_{\partial\Omega}=h$. Thus, we can write
\[
h(f)(c)=df(\Omega)[V]=\int_{\partial\Omega}g\, \alpha\, ds
\]
if $h=\alpha\vec{n}$. Also, the gradient in terms of a Riemannian representation of the shape derivative in terms of the metric $G^A$ can be written as
\[
\mbox{grad}f=\frac{1}{1+A\kappa_c^2}g
\]
The essential operation in Riemannian geometry is the covariant derivative $\nabla_hk$ which is a directional derivative of vector fields in terms of tangential vectors such that  $\nabla_hk\in T_cB_e$, if $h,k \in T_cB_e$. \begin{revision}Often in differential geometry\end{revision}, the covariant derivative is written in terms of the Christoffel symbols. In \cite{MM-2006} explicit expressions for the Christoffel symbols are derived in terms of the Riemannian metric $G^A$. However, in order to reveal the relation with the shape calculus, we show another representation of the covariant derivative in theorem \ref{covariant}.
\begin{theorem}\label{covariant}
Let $\Omega\in\R^2$ be a shape and $V,W\in C^\infty(\R^2,\R^2)$ vector fields which are orthogonal at the boundary $\partial\Omega$, i.e.,
$V|_{\partial\Omega}=\alpha\vec{n}$ with $\alpha:=\scp{V|_{\partial\Omega},\vec{n}}$, and $W|_{\partial\Omega}=\beta\vec{n}$ with $\beta:=\scp{ W|_{\partial\Omega},\vec{n}}$, such that
$h:=\alpha\vec{n}$, $k:=\beta\vec{n}$ belong to the tangent space of $B_e$.  Then, the covariant derivative \begin{revision}associated with the Riemmanian metric $G^A$ \end{revision}can be expressed as
\begin{align*}
\nabla_{V}W&:=\nabla_{h}k=\frac{\partial\beta}{\partial\vec{n}}\alpha+
\frac12(\kappa_c+\frac{2A\kappa_c^3}{1+A\kappa_c^2})\alpha\beta+A\kappa_c(\alpha\beta)_{\tau\tau}\\
&=\scp{\begin{revision}\mathrm{D}\end{revision}W\,V,\vec{n}}+\frac12
(\kappa_c+\frac{2A\kappa_c^3}{1+A\kappa_c^2})\scp{V,\vec{n}}\scp{W,\vec{n}}
+A\kappa_c(\scp{V,\vec{n}}\scp{W,\vec{n}})_{\tau\tau}
\end{align*} 
where expressions ``$(.)_{\tau\tau}$'' mean second order derivative in unit tangential direction of the shape boundary \begin{revision}and the notation ``$\mathrm{D}W\,V$'' means the directional derivative of the vector field $W$ in the direction $V$\end{revision}.
\end{theorem} 
%%%
\begin{proof}
\begin{revision}The strategy which leads to the expression for $\nabla_{V}W$ above is to exploit \end{revision}the product rule for Riemannian connections, $hG^A(k,\ell)=G^A(\nabla_{h}k,\ell)+G^A(k,\nabla_{h}\ell)$ for any $\ell\in T_cB_e$. \begin{revision}The left hand side of the product rule is expressed in more details which are then grouped in two terms for the right hand side.\end{revision}
We assume that $Z\in C^\infty(\R^2,\R^2)$ is a vector field with $\ell:=Z|_{\partial\Omega}=\gamma\vec{n}$. Then
\begin{revision}the application of shape calculus rules for volume and boundary functionals as in \cite{Delfour-Zolesio-2001} gives \end{revision}
\begin{align}
h(G^A(k,\ell))=&d\left(\int_{\partial\Omega}(1+A\kappa_c^2)\beta\gamma ds\right)[V]\\
=&
\int_{\partial\Omega}\frac{\partial[(1+A\kappa_c^2)\beta\gamma]}{\partial\vec{n}}\alpha+
\kappa_c(1+A\kappa_c^2)\alpha\beta\gamma ds\\
=&\label{eq-2.3}
\int_{\partial\Omega}2A\kappa_c(\kappa_c^2\alpha+\alpha_{\tau\tau})\beta\gamma+(1+A\kappa_c^2)\frac{\partial\beta}{\partial\vec{n}}\gamma\alpha\\
&\qquad +\nonumber
(1+A\kappa_c^2)\frac{\partial\gamma}{\partial\vec{n}}\beta\alpha+\kappa_c(1+A\kappa_c^2)\alpha\beta\gamma ds
\\
=&\label{michor-eq}
\int_{\partial\Omega}2A\kappa_c^3\alpha\begin{revision}\beta\gamma\end{revision}+2A\kappa_c\alpha(\beta\gamma)_{\tau\tau}+(1+A\kappa_c^2)\frac{\partial\beta}{\partial\vec{n}}\gamma\alpha\\
&\qquad +\nonumber
(1+A\kappa_c^2)\frac{\partial\gamma}{\partial\vec{n}}\beta\alpha+\kappa_c(1+A\kappa_c^2)\alpha\beta\gamma ds
\end{align}
\begin{revision} where we note in equation (\ref{eq-2.3}) that 
$(\partial \kappa_c/\partial\vec{n})\alpha=\kappa_c^2\alpha+\alpha_{\tau\tau}$,\end{revision} which is an immediate consequence of equation (7) in section 2.2 of \cite{MM-2006} \begin{revision} and in equation (\ref{michor-eq}) that the integrands $\alpha_{\tau\tau}\beta\gamma$ and $\alpha(\beta\gamma)_{\tau\tau}$ give the same integral because of the periodicity of $\partial\Omega$.
The expression in equation (\ref{michor-eq}) is now split in two mirror symmetric parts according to the product rule for Riemannian connections and written in the form of scalar products with $G^A$, thus causing the denominator $(1+A\kappa_c^2)$. It remains to show $C^\infty$ linearity, the Leibniz rule and symmetry (torsion free). We sketch symmetry here:
\begin{align*}
\nabla_{h}k-\nabla_{k}h=&\frac{\partial\beta}{\partial\vec{n}}\alpha+
\frac12(\kappa_c+\frac{2A\kappa_c^3}{1+A\kappa_c^2})\alpha\beta+A\kappa_c(\alpha\beta)_{\tau\tau}\\
-&
\left(\frac{\partial\alpha}{\partial\vec{n}}\beta+
\frac12(\kappa_c+\frac{2A\kappa_c^3}{1+A\kappa_c^2})\beta\alpha+A\kappa_c(\beta\alpha)_{\tau\tau}\right)\\
=&\frac{\partial\beta}{\partial\vec{n}}\alpha -\frac{\partial\alpha}{\partial\vec{n}}\beta =
{h}(k)-{k}(h)=[{h},{k}]
\end{align*}
where ``$[.,.]$'' denotes the Lie bracket. In the penultimate equation, it should be noted that most of the terms in ${h}(k)-{k}(h)$ cancel out, just leaving over only $\frac{\partial\beta}{\partial\vec{n}}\alpha -\frac{\partial\alpha}{\partial\vec{n}}\beta$.
\end{revision}
\end{proof}

The Riemannian connection now gives the means to investigate the Hessian of an objective defined on a shape, resp. its boundary. As in \cite{Absil-book-2008}, we define the Riemannian Hessian of a function $f$ at the point $c\in B_e$ as the linear mapping of $T_cB_e$ to itself defined by
\[
\mbox{Hess}f(c)[h]:=\nabla_h\mbox{grad}f
\]
In the terminology of shape optimization and with the definition of vector fields as in theorem \ref{covariant} and the identification of the boundary $c=\partial\Omega$ with its shape, we may identify this with a now so-called Riemannian shape Hessian.
\begin{definition}For the setting of theorem \ref{covariant}, we define the {\em Riemannian shape Hessian} as
\end{definition}
\[
\mbox{Hess}f(\Omega)[V]:=\nabla_V\mbox{grad}f
\]
The next theorem gives a correlation of the Riemannian shape Hessian with the standard shape Hessian which is defined by repeated shape differentiation.\begin{revision} It can be found, e.g., in the book \cite{Absil-book-2008} for the finite dimensional case. But at the level of abstraction on the current point of the paper, the arguments developed in \cite{Absil-book-2008} can be applied identically. The novelty here lies in the interpretation of $d(df(\Omega)[W])[V]$ as the ``classical'' shape Hessian.\end{revision}
\begin{revision}
\begin{theorem}\label{RSH}
The Riemannian shape Hessian based on the Riemannian metric $G$ satisfies the relation
\[
G(\mbox{Hess}f(\Omega)[V],W)=d(df(\Omega)[W])[V]-df(\Omega)[\nabla_VW]
\]
where $V,W$ are defined as in theorem \ref{covariant}  and $d(df(\Omega)[W])[V]$ denotes the standard shape Hessian as defined in \cite{Delfour-Zolesio-2001}. Furthermore, we observe the following symmetry
\[
G(\mbox{Hess}f(\Omega)[V],W)=G(V,\mbox{Hess}f(\Omega)[W]) 
\]
\end{theorem}
\begin{proof}
The identical arguments as in the proofs of propositions 5.5.2 and 5.5.3 of \cite{Absil-book-2008} can be used.\end{proof}
\end{revision}

Now we can use the Riemannian shape Hessian in order to formulate a Taylor series expansion as well as optimality conditions. Since the subsequent theorem holds in more general cases than just for $B_e$, we formulate it in a more general notation. \begin{revision}Here, the key tool is the exponential map $\exp$, locally identifying the tangent space of a manifold with the manifold itself. It is defined in the usual differential geometric sense, which is also possible for shape manifolds, e.g.~\cite{MM-2006}.\end{revision}

\begin{theorem}\label{Taylor}
Let $({\cal N},G)$ be a Riemannian manifold with metric $G$ and norm $\norm{.}:=G(.,.)$. 
\begin{rrev}Let\end{rrev} the set $U\subset{\cal N}$ be a convex subset of ${\cal N}$. We consider all $x\in{\cal N}$ and $\Delta x\in T_x{\cal N}$ with $x,\exp_x(\Delta x)\in U$  and denote the parallel transport along the geodesic $\gamma:[0,1]\to \cN, t\mapsto \gamma(t):=\exp_x(t\Delta x)$ by $\cP_{\alpha,\beta}:T_{\gamma(\alpha)}\cN\to T_{\gamma(\beta)}\cN$.
We assume for the Riemannian Hessian of a function $f:U\to \R$  the following Lipschitz property at $x\in {\cal N}$:
\[
\|\cP_{1,0}\hess{}\, f(\exp_x(\Delta x))\cP_{0,1}-\hess{}\, f(x)\| \le L\|\Delta x\|\, , \quad\forall\, \exp_x(\Delta x)\in U
\]
with a constant $L<\infty$. Then, we \begin{rrev}arrive at\end{rrev} the estimation
\[
|f(\exp_x(\Delta x))-f(x)+G(\grad f(x),\Delta x)+\frac12G(\hess{}\, f(x)\Delta x,\Delta x)|\le\frac{L}{6}\|\Delta x\|^3
\]
\end{theorem}
\begin{proof}
Let us consider the mapping $\varphi:[0,1]\ni t\mapsto f(\exp_x(t\Delta x))$.
We note that for all differentiable functions and in particular for
$\varphi$ \begin{rrev}holds true\end{rrev}
\[
\int\limits_0^1\int\limits_0^t\varphi''(s)-\varphi''(0)ds\, dt =
\varphi(1)-\varphi(0)-\varphi'(0)-\frac12\varphi''(0)
\]
Since
\begin{align*}
\varphi(1)&=f(\exp_x(\Delta x))\, , \ \varphi(0)=f(x)\, , \ \varphi'(0)=G(\grad f(x),\Delta x)\, 
\\
\varphi''(0)&=G(\hess{}\, f(x)\Delta x,\Delta x)
\end{align*}
we observe
\begin{eqnarray*}
&\quad& |f(\exp_x(\Delta x))-f(x)+G(\grad f(x),\Delta x)+\frac12G(\hess{}\, f(x)\Delta x,\Delta x)| \\
&=&
\int\limits_0^1\int\limits_0^t|\varphi''(s)-\varphi''(0)|ds\, dt \\
&=& \int\limits_0^1\int\limits_0^t\left|\left(\left(\cP_{s,0}\hess{}\, f(\exp_x(s\Delta x))\cP_{0,s}-\hess{}\, f(x)\right)\Delta x,\Delta x\right)\right|ds\, dt\\
&\le& \int\limits_0^1\int\limits_0^t\|\cP_{s,0}\hess{}\, f(\exp_x(s\Delta x))\cP_{0,s}-\hess{}\, f(x)\|\,\|\Delta x\|^2 ds\, dt\\
&\le& \int\limits_0^1\int\limits_0^tsL\|\Delta x\|^3 ds\, dt = \frac{L}6\|\Delta x\|^3
\end{eqnarray*}
\end{proof}
 
Now, we can exploit the Taylor expansion of Theorem  \ref{Taylor} for
necessary and sufficient optimality conditions\index{optimality conditions}
\begin{theorem}\label{2nd-order} Under the assumptions of Theorem  \ref{Taylor} we obtain:
\begin{itemize}
\item[(a)] If $\hat{x}$ is an optimal solution, then $\hess{}\, f(\hat{x})\ge 0$, i.e.
           $\begin{revision}G\end{revision}(\hess{}\, f(\hat{x})h,h)\ge 0$, for all $h\in T_{\hat{x}}{\cal N}$
\item[(b)] If $\hat{x}$ satisfies $\grad f(\hat{x})=0$, and $\hess{}\, f(\hat{x})$ is coercive, i.e.
           $\begin{revision}G\end{revision}(\hess{}\, f(\hat{x})h,h)\ge c\|h\|^2$, for all $h\in T_{\hat{x}}{\cal N}$ and for some $c>0$, then
           $\hat{x}$ is a local minimum, provided \begin{rrev}that\end{rrev} $\hess{} f(\hat{x})$ satisfies a Lipschitz condition
           as in Theorem  \ref{Taylor}.
\end{itemize}
\end{theorem}
\begin{proof}
The proof is identical with the standard proof in linear spaces.
\end{proof}

\medskip

\begin{example}\normalfont
Now let us study a particular example, one of the simplest but nevertheless instructive shape optimization problems 
\[
\min\limits_{\Omega\subset\R^2}f(\Omega):=\int_\Omega\psi(x)dx
\]
where $\psi:\R^2\to\R$ is a sufficiently smooth scalar valued function. We use the vector field definitions of theorem \ref{covariant}. The shape derivative of this objective and the $G^A$-gradient are
\begin{align*}
df(\Omega)[V]&=\int_{\partial\Omega}\psi\scp{V,\vec{n}} ds\\
\mbox{grad}f(c)&=\frac{\psi}{1+A\kappa_c^2}\, ,\qquad c=\partial\Omega
\end{align*}
\begin{revision}With the notation of theorem \ref{covariant}\end{revision}, the standard shape Hessian is computed by shape differentiating the shape derivative
\begin{align*}
d(df(\Omega)[W])[V]&=d(\int_{\partial\Omega}\psi k ds)[h]=
\int_{\partial\Omega}\frac{\partial(\psi k) }{\partial\vec{n}}h+\kappa_c \psi khds\\
&=\int_{\partial\Omega}\frac{\partial\psi}{\partial\vec{n}}kh+\psi\frac{\partial k}{\partial\vec{n}}h+\kappa_c \psi khds\\
&=\int_{\partial\Omega}(\frac{\partial\psi}{\partial\vec{n}}+\kappa_c\psi)\scp{W,\vec{n}}\scp{V,\vec{n}}+\psi \scp{DW\, V,\vec{n}}ds
\end{align*}
\begin{rrev}It should be noted that again $k$ and $h$ are the function representatives of $W$ and $V$.\end{rrev}
We observe that the standard shape Hessian is not symmetric. In contrast to that, the Riemannian shape Hessian computed by application of theorem \ref{RSH} is given by
\begin{align}\nonumber
G^A(\mbox{Hess}f&(\Omega)[V],W)=d(df(\Omega)[W])[V]-df(\Omega)[\nabla_VW]\\\nonumber
&=\int_{\partial\Omega}(\frac{\partial\psi}{\partial\vec{n}}+\kappa_c\psi)\scp{W,\vec{n}}\scp{V,\vec{n}}+\psi \scp{DW\, V,\vec{n}}ds -\int_{\partial\Omega}\psi\scp{\nabla_VW,\vec{n}} ds\\\nonumber
&=\int_{\partial\Omega}(\frac{\partial\psi}{\partial\vec{n}}+\kappa_c\psi)\scp{W,\vec{n}}\scp{V,\vec{n}}+\psi \scp{DW\, V,\vec{n}}ds\\\nonumber
&\hspace{1.0cm}-\int_{\partial\Omega}\psi(\scp{DW\,V,\vec{n}}+\frac{\psi}2
(\kappa_c+\frac{2A\kappa_c^3}{1+A\kappa_c^2})\scp{V,\vec{n}}\scp{W,\vec{n}}
\\\nonumber
&\hspace{1.0cm}+\psi A\kappa_c(\scp{V,\vec{n}}\scp{W,\vec{n}})_{\tau\tau}ds\\\label{R-Hess}
&=\int_{\partial\Omega}(\frac{\partial\psi}{\partial\vec{n}}+\frac{\kappa_c}{2}\psi
-\frac{A\kappa_c^3}{1+A\kappa_c^2}\psi)\scp{V,\vec{n}}\scp{W,\vec{n}}-\psi A\kappa_c(\scp{V,\vec{n}}\scp{W,\vec{n}})_{\tau\tau}ds
\end{align}
Now, as already abstractly shown in corollary \ref{RSH}, we observe symmetry of the Riemannian shape Hessian also in this example. We will study this example in more specific details in section \ref{numex}.
\end{example}

\section{Convergence of Riemannian Newton Methods}\label{sec:Newton}
Now we formulate a contraction result for Newton iterations on manifolds which is in line with linear space theorems in \cite{BockHxG1987a,Deufl-Newton} \begin{revision} for Newton-like methods
\begin{equation}\label{Newton}
x^{k+1}=x^k-M_kF(x^k)
\end{equation}
where $M_k$ is an approximation of the inverse of the derivative of the function $F$, which defines a root finding problem $F(x)=0$\end{revision}. 
We \begin{rrev}prove\end{rrev} convergence properties for Newton-like methods on Riemannian manifolds, but have always in mind that in this paper we want to solve a particular root finding problem for the gradient of an objective \begin{rrev}functional\end{rrev}. Thus, every time the Jacobian is mentioned, we can think it as the Jacobian of the gradient and thus the Hessian of an objective on a Riemannian manifold.

\begin{theorem}\label{contraction}
We consider a complete Riemannian Manifold $(\cN,G)$ with norm $\norm{.}:=G(.,.)$ on the tangential bundle $T\cN$. The set $D\subset\cN$ is assumed to be \begin{rrev}simply\end{rrev} connected and open. We are seeking a singular point of the twice differentiable vector field $F:\cN\to T\cN$ by employing a Newton method on manifolds. The symbol $J$ denotes the covariant derivative of $F$ such that $J(x)v:=\nabla_vF_x$. Furthermore, we assume that for all points $x,y\in D$ with $y=\exp_x(\Delta x), \Delta x:=-M(x)F(x)$ and $M(x)\in\End(T_x\cN))$ and invertible  and all $t\in [0,1]$, the following Lipschitz conditions are satisfied along the geodesic $\gamma:[0,1]\to \cN, t\mapsto \gamma(t):=\exp_x(-tM(x)F(x))$ with parallel transport $\cP_{\alpha,\beta}:T_{\gamma(\alpha)}\cN\to T_{\gamma(\beta)}\cN$.
\begin{itemize}
\item[(1)] There exists $\omega<\infty$ with
\[
\norm{M(y)(\cP_{t,1}J(\gamma(t))\cP_{0,t}-\cP_{0,1}J(x))\Delta x)}\le\omega t \norm{\Delta x}^2
\]
\item[(2)] There is a constant upper limit $\kappa <1$ for the function $\tilde{\kappa}(x)$ in 
\[
\norm{M(y)\cP_{0,1}(F(x)+J(x)\Delta x)}=: \tilde{\kappa}(x)\norm{\Delta x}
\]
with $\tilde{\kappa}(x)\le \kappa$.
\end{itemize}
If $x_0$ satisfies $\delta_0<1$, where $\delta_k:=\kappa+\frac{\omega}{2}\norm{M(x_k)F(x_k)}, k=0,1,\ldots$, then follows
\begin{itemize}
\item[(1)] The iteration $x_{k+1}:=\exp_{x_k}(-M(x_k)F(x_k))$ is well defined and stays in $D_0:=\{x\in D\,|\, d(x,x_0)\le\norm{M(x_0)F(x_0)}/(1-\delta_0)\}$.
\item[(2)] There exists $\hat{x}\in D_0$ with $\lim\limits_{k\to\infty}x^k=\hat{x}$ in the sense that 
$\lim\limits_{k\to\infty} d_{G}(x^k,\hat{x})=0$, where $d_G$ is the geodesic distance.
\item[(3)] There holds the a priori estimation
\[
d_{G}(x^{k},\hat{x})\le\frac{\delta_k}{1-\delta_k}\norm{\Delta x_k}\le \frac{\delta_0^k}{1-\delta_0}\norm{\Delta x_k}
\]
where $\Delta x_k:=-M(x^k)F(x^k)$.
\item[(4)] There yields the contraction property
\[
\norm{\Delta x_{k+1}}\le\delta_k\norm{\Delta x_k}=(\kappa+\frac{\omega}{2}\norm{\delta x_k})\norm{\Delta x_k}
\]
\item[(5)] If the mapping $x\mapsto M(x)$ is continuous and $M(\hat{x})$ is nonsingular, then $\hat{x}$ is not only a fixed point but rather a root of the equation $F(x)=0$. 
\end{itemize}
\end{theorem}
\begin{proof}The proof follows the standard lines---now in manifold notation. First, we note the manifold variant of the fundamental theorem of calculus along the geodesic $\gamma$ for any $C^1$ vector field $X$ (cf.~\cite{Ferreira-2002}):
\[
X(\gamma(t))=\cP_{0,t}X(\gamma(0)+\int\limits_0^t\cP_{s,t}\nabla_{\dot{\gamma}(s)}X(\gamma(s))ds = 
\cP_{0,t}X(\gamma(0)+\int\limits_0^t\cP_{s,t}DX(\gamma(s))\dot{\gamma}(s)ds
\]
and have in mind that obviously $\dot{\gamma}(s)=\cP_{0,s}\Delta x^k$. 
We show the contraction property and use the abbreviation $R(x^k):=F (x^k) - J (x^k) M(x^k)F (x^k)$
\begin{align*}
\norm{\Delta x^{k+1}} &= \norm{M(x^{k+1})F(x^{k+1})}\\
&=\norm{M (x^{k+1}) \cP_{0,1}R(x^k) + {\color{black}M (x^{k+1}) }\bigl[ F (x^{k+1}) - \cP_{0,1}R(x^k) \bigr]}\\
&\le \kappa\norm{\Delta x^k}+\norm{M(x^{k+1})\int\limits_0^1\cP_{t,1}\begin{revision}J\end{revision}(\gamma(t))\cP_{0,t}\Delta x^k-\cP_{0,1}J(x^k)\Delta x^kdt}\\
&\le \kappa\norm{\Delta x^k}+\int\limits_0^1\norm{M(x^{k+1})(\cP_{t,1}\begin{revision}J\end{revision}(\gamma(t))\cP_{0,t}-\cP_{0,1}J(x^k))\Delta x^k}dt\\
&\le \kappa\norm{\Delta x^k}+\int\limits_0^1 t\omega\norm{\Delta x^k}^2dt=\kappa\norm{\Delta x^k}+\frac{\omega}{2}\norm{\Delta x^k}^2=
\delta_k\norm{\Delta x^k}
\end{align*}
Now we conclude inductively that $\delta_k<1$ for all $k$ and thus the series $\{\delta_k\}_{k=0}^\infty$ and $\{\norm{\Delta x_k}\}_{k=0}^\infty$ are montonically decreasing. Now, we show that the series $\{x_k\}_{k=0}^\infty$ stays in $D_0$. We use the triangle inequality which holds because of the metric properties of the Riemannian metric.
\begin{align}\label{initial.1}
d_G(x^{k},x^0) &=d_G(\exp(\Delta x_{k-1})\circ\ldots\circ\exp(\Delta x_{1})\circ\exp_{x_0}(\Delta x_0),x_0)\\\label{initial.2}
 &\le \sum\limits^{k-1}_{j=1} \norm{ \Delta x^j} \le \sum\limits^{k-1}_{j=1} \delta_0^j\norm{ \Delta x^0}\le
 \frac{\norm{\Delta x_0}}{1-\delta_0}\qquad   \mbox{(geometric series)} 
\end{align}
Analogously, we show the Cauchy property of the series $\{x_k\}_{k=0}^\infty$. 
\begin{align} \label{Cauchy.1}
d_G(x^m,x^n)  &= d(\exp(\Delta x_{m-1})\circ\ldots\circ\exp(\Delta x_{n+1})\circ\exp_{x_n}(\Delta x_n),x_n) \\\label{Cauchy.2}
            &\leq \sum\limits^{m-1}_{k=n} \norm{ \Delta x^k} 
                            \leq \sum\limits^{m-1}_{k=n} \, \delta_n^k  
                            \norm{\Delta x^n}                              
                       \leq   \sum\limits^{m-1}_{k=n} \, \delta_0^k  
                            \norm{ \Delta x^0} \\
&\leq \varepsilon\ \mbox{for any }\  \varepsilon>0\mbox{, if }n, m \to \infty 
\end{align}
The a priori estimation is yet another application of the triangle inequality. 
Now, if $ M (.) $ is continuous and $ x^k \to \hat{x}$, we pass the defining equation
\[
x_{k+1}:=\exp_{x_k}(-M(x_k)F(x_k))
\] to the limit, having in mind that $\exp$ is continuous and $\exp_y(z)=y$ implies $z=0$ for any $y$. Thus, we observe that
$M(\hat{x})F(\hat{x})=0$. Finally, if $M(\hat{x})$ is nonsingular, we conclude that $F(\hat{x})=0$

\end{proof}

Now, we proof quadratic convergence for the exact Newton method--again completely parallel to the discussions in
\cite{BockHxG1987a,Deufl-Newton}.
\begin{corollary}\label{quadratic} Together with the assumptions of theorem \ref{contraction}, we choose $M(x)=\begin{revision}J\end{revision}(x)^{-1}$, thus defining the exact Newton method on manifolds. The resulting iteration converges locally quadratically, i.e., there is a $\tilde{k}\in \N$ and a $C<\infty$ such that 
\[
d_G(x^{k+1},\hat{x})\le Cd_G(x^{k},\hat{x})^2\, , \quad\forall k\ge\tilde{k}
\]
where $d_G$ denotes again the geodesic distance.
\end{corollary}
\begin{proof} Because of the choice $M(x)=\begin{revision}J\end{revision}(x)^{-1}$, we observe for $\kappa$ in theorem \ref{contraction} from 
\[
\norm{M(y)\cP_{0,1}(F(x)+J(x)\Delta x)}=\norm{M(y)\cP_{0,1}(F(x)-J(x)J(x)^{-1}F(x))}=0
\]
And thus $\kappa=0$.
Since $\{\delta_k\}_{k=0}^\infty$ is monotonically decreasing to zero, there is $\tilde{k}$ such that $\delta_k=\frac{\omega}{2}\norm{\Delta{x^k}}\le\frac14$ for all $k\ge \tilde{k}$, which implies also $\delta_k/(1-\delta_k)\le 1/3$ for all $k\ge \tilde{k}$.
From the fact that the exponential function is an isometry and from the a priori estimation in theorem \ref{contraction}, we observe
\begin{align*}
\norm{\Delta{x^k}}&=d_G(x^{k+1}, x^k)\le d_G(x^{k+1}, \hat{x})+d_G(x^{k}, \hat{x})\le\frac{\delta_k}{1-\delta_k}\norm{\Delta x^k}+d_G(x^{k}, \hat{x})\\
&\le \frac13\norm{\Delta x^k}+d_G(x^{k}, \hat{x})
\end{align*}
and therefore $\norm{\Delta x^k}\le\frac32d(x^{k}, \hat{x})$. Now, we use again the a priori estimation
\[
d_G(x^{k+1},\hat{x})\le \frac{\delta_{k+1}}{1-\delta_{k+1}}\norm{\Delta x_{k+1}}\le\frac43\frac{\omega}{2}\norm{\Delta x^k}^2
\le\frac32\omega d_G(x^{k},\hat{x})^2
\]
\end{proof}

In shape optimization, it is very rare that one can get hold of $J(x^k)$ (i.e., the Hessian of an objective) away from the optimal solution. However, in many cases, expressions can be derived, which deliver the exact Hessian, if evaluated at the solution. That means that often the situation occurs that an approximation $M_k^{-1}$ of $J(x^k)$ is available with the property $M_k^{-1}\to \begin{revision}J\end{revision}(\hat{x}), k\to\infty$. We show local superlinear convergence in those cases. 
\begin{corollary}\label{superlinear}Together with the assumptions of theorem \ref{contraction}, we choose $M_k:=M(x^k)$ such that 
\[
\frac{\norm{[M_k^{-1}-\begin{revision}J\end{revision}(x^k)]\Delta x^k}}{\norm{\Delta x^k}}\to 0, \quad k\to\infty
\]
\begin{revision}If there is a constant $C<\infty$ such that the approximation $M_k$ is uniformly bounded, $\norm{M_k}\le C, \forall k$, then the resulting Newton-like iteration (\ref{Newton}) converges locally superlinearly.\end{revision}
\end{corollary}
\begin{proof} We observe for $\kappa_k:=\kappa(x^k)$  in theorem \ref{contraction}
\begin{align*}
\kappa_k&=\frac{\norm{M_{k+1}\cP_{0,1}(F(x^k)+J(x^k)\Delta x^k)}}{\norm{\Delta x^k}}=
\frac{\norm{M_{k+1}\cP_{0,1}(M_k^{-1}+J(x^k))\Delta x)}}{\norm{\Delta x^k}}\\
&\le\frac{\norm{M_{k+1}}\norm{\cP_{0,1}(M_k^{-1}+J(x^k))\Delta x)}}{\norm{\Delta x^k}}\to 0
\end{align*}
We use again the apriori estimation in an analogous fashion as in the proof of corollary \ref{quadratic} to obtain
\[
d_G(x^{k+1},\hat{x})\le \frac{\delta_{k+1}}{1-\delta_{k+1}}\frac{1-2\delta_k}{1-\delta_k}d_G(x^{k+1},\hat{x})
\le \frac{1-2\delta_k}{(1-\delta_k)^2}\delta_k d_G(x^{k+1},\hat{x})
\]
Since $0\le\delta_k^2$, we observe that $(1-2\delta_k)/(1-\delta_k)^2\le 1$ which \begin{revision} gives\end{revision}
\[
d_G(x^{k+1},\hat{x})\le\delta_k d_G(x^{k},\hat{x})\mbox{ with } \delta_k\to 0,\quad k\to \infty
\]
\end{proof}

{\sc Remarks:} 
\begin{itemize}
\item[-] If we can assume that the quality of the Hessian approximation satisfies even 
$\norm{M_k^{-1}+J(x^k)}\le C\norm{\Delta x^k}$ then we observe also quadratic convergence as an immediate consequence of an obvious refinement of the proof of corollary \ref{superlinear}.
\item[-] Note that the condition in corollary \ref{superlinear} is similar to the Dennis-Mor\'e condition \cite{Dennis-More-1974,Byrd-Nocedal_1989}, which is also applicable for quasi-Newton update techniques on Riemannian manifolds \begin{revision}\cite{Qi-2011,Absil-Dennis-More}\end{revision}.
\end{itemize}

The application of the exponential mapping within the Newton method is an expensive operation. It is recommendable to replace this step by a so-called {\em retraction} mapping 
\cite{Absil-book-2008,Udriste-1994,Qi-2011}, i.e., a smooth mapping $r_x:T_xB_e\to B_e$ with the following properties: 
\begin{itemize}
\item[a)]$r_x(0)=x$
\item[b)]$Dr_x(0)=id_{T_xB_e}$ (local rigidity condition \cite{Absil-book-2008})
\end{itemize}
with the canonical identification $T_0T_xB_e\cong T_xB_e$. Properties a) and b) are also satisfied by the exponential mapping. An example of this kind of mapping is the following mapping that we will use in our implementation of Newton optimization methods on $B_e$:
\begin{align*}
r_x:\ &T_xB_e\to B_e\\
    &\eta\mapsto \left(\begin{array}{rl}
                   x+\eta:& S^1\to\R^2\\
                          &  \theta\mapsto x(\theta)+\eta(\theta)
                   \end{array}\right) 
\end{align*}
The implementation of the retraction $r_x$ defined above is much simpler than the implementation of the exponential map. However, one should keep in mind that \begin{rrev}this retraction may leave $B_e$ if $\eta$ is not small enough, as intersections and kinks may appear in the shape. \end{rrev}

Now, the Newton method for optimization is generalized to
\begin{equation}\label{gen-Newton}
x^{k+1}=r_{x^k}(-\begin{rrev}M_k\end{rrev}F(x^k))
\end{equation}
where the special case $r_x=\exp_x$ considered above falls also into this class of iterations. Since any retraction is a smooth mapping and because of the local rigidity, we observe that
\[
d_G(x,r_x(\eta))=d_G(x,\exp_x(\eta))+cd_G(x,\exp_x(\eta))^2
\]
for some constant $c$ and for small enough tangential vectors $\eta$. Plugging this into the estimations  (\ref{initial.1}, \ref{initial.2}) and (\ref{Cauchy.1}, \ref{Cauchy.2}), we conclude that locally linear, superlinear and quadratic convergence properties are not changed by using more general retractions rather than the special case of the exponential map, \begin{rrev}as has been similarly observed in \cite{Absil-book-2008}. \end{rrev}

\section{Numerical Experiments}\label{numex}
Here, we study the linear and quadratic convergence properties of standard optimization algorithms applied to an example for shape optimization in $\R^2$ which is as simple as possible, but nevertheless reveals structures reflecting the discussion of this paper. We consider the following problem:
\begin{align}\label{MSO}
\min\limits_{\Omega}f(\Omega)=\int_{\Omega}x_1^2+\mu\begin{revision}^2\end{revision} x_2^2-1dx\, , \qquad \mu\ge 1
\end{align}
\begin{revision}Thus, we carry on example 1 above with specifically $\psi(x)=x_1^2+\mu^2 x_2^2-1$. 
\begin{rrev}At first glance, the reader might be misled by the apparent familiarity\end{rrev} of the problem structure. It \begin{rrev}resembles\end{rrev} a quadratic optimization problem, whose solution is expected at the point $(0,0)$ or maybe the empty set. The objective value of both (zero and the empty set) is zero and, as we see below, the optimal objective value is strictly less than zero. The rationale for finding the optimal solution is to find a shape, where the shape derivative is zero, i.e., where the integrand above is zero. This is \end{revision} the shape $\hat{\Omega}$ with boundary
\[
\partial\hat{\Omega}=\{x\in\R^2\ |\ x_1^2+\mu^2 x_2^2=1\}
\]
\begin{revision}where $\psi(x)=0$ in example 1 above and therefore, all terms  related to $\psi$ in (\ref{R-Hess}) vanish\end{revision}. We observe that the Riemannian shape Hessian at this solution is a multiplication operator
\begin{revision}
\begin{align}\label{numhess}
\hess{}f(\hat{\Omega})V&=\nu\cdot V\\
\nu(s)&=2(s_1n_1(s)+\mu^2s_2n_2(s))=2\sqrt{s_1^2+\mu^4s_2^2}\in[2,2{\mu}], s\in\partial\hat{\Omega}
\end{align}
where $(n_1(s),n_2(s))$ denotes the unit normal at $s\in\partial\hat{\Omega}$.
\end{revision}
Therefore, the Hessian is coercive and thus the solution is locally unique. 
\begin{revision} Furthermore, we observe for an exact Newton method $\kappa=0$ in theorem \ref{contraction} and also $\omega=\max\{2/\nu(s)\, |\, s\in\partial\hat{\Omega}\}$ at the solution. Together with obvious smoothness properties, we can apply theorem \ref{contraction} and expect quadratic convergence. \begin{rrev}By exact Newton method, we mean here the iteration (\ref{gen-Newton}) with $M_k$ replaced by the exact Hessian (\ref{numhess}). This is the method that we use in the comparisons below.\end{rrev}

\end{revision}

The distance of a shape iterate $\Omega^k$ to the optimal shape $\hat{\Omega}$ should be measured by the length of a connecting geodesic. However, in order to reduce the numerical effort, we exploit the fact that a connecting geodesic is also the result of an exponential map, which on the other hand is an isometry. Thus, we use for shapes close enough to the optimal solution such that
$\Omega^k=\hat{\Omega}+h$ with $h=\alpha\vec{n}|_{\partial\hat{\Omega}}$ as second order accurate substitute for the geodesic distance the measure
\begin{equation}\label{distapprox}
\tilde{d}(\Omega^k,\hat{\Omega})=\int\limits_{\partial\hat{\Omega}}|\alpha|ds
\end{equation}
i.e., just the absolute value of the area between the iterates and the optimal shape.

\def\atan{\mathrm{atan2}}
\def\mumat{\left[\begin{array}{cc}1&0\\0&\mu\end{array}\right]}
Because of the simplicity of the objective (\ref{MSO}), the integrations for the evaluations of the objective can be performed in the form of integration along the shape boundary in polar coordinates \begin{revision}  which can be evaluated by usage of the trapezoidal rule for a piecewise linear boundary discretization in the following way.
We define the transformed domain $\Omega^{\mu}:=\mumat\Omega$ with \begin{rrev}boundary\end{rrev} $c^\mu:=\partial\mumat\Omega$ and resulting change of variables $y:=\mumat^{-1}x$. Then, we compute
\begin{align*}
\int_{\Omega}x_1^2+\mu^2x_2^2 -1 dx &=\frac1\mu \int_{\Omega^\mu}y_1^2+y_2^2 -1 dx =
\frac1\mu \int_{S^1}\int_0^{\norm{c^\mu(\theta)}}(r^2-1)rdr d\theta\\
&=\frac1\mu\int_0^{2\pi}\int_0^{\norm{c^\mu(\theta)}}r^3-r dr d\theta =
 \frac1\mu\int_0^{2\pi}\frac{\norm{c^\mu(\theta)}^4}4 -\frac{\norm{c^\mu(\theta)}^2}2 {d\theta}\\
\approx\frac1{2\mu}\sum_{i=1}^N(\atan(c^\mu_{i+1})&-\atan(c^\mu_{i}))
\left(\frac{\norm{c^\mu_{i+1}(\theta)}^4+\norm{c^\mu_i(\theta)}^4}4 -\frac{\norm{c^\mu_{i+1}(\theta)}^2+\norm{c^\mu_i(\theta)}^2}2\right)
\end{align*}
In the last line, the curve $c^\mu$ is approximated by a periodic polygon with nodes $c^\mu_i\in\R^2$, $i=1,\ldots N+1$, and $c^\mu_{N+1}=c^\mu_{1}$. Furthermore, the function $\atan:\R^2\to (-\pi,\pi]$ means the standard inverse tangent function with two arguments as provided by many programming languages (C, C++, FORTRAN, Python,..).
\end{revision}

Similarly, in the line of (\ref{distapprox}), first order approximations of the distance of curves to the optimal solution, which differ only by a small amount are computed by
\begin{revision}
\begin{align*}
\bar{d}(\Omega^k,\hat{\Omega})&=\int_{\partial\hat{\Omega}}\left|\;\norm{c(s)}-\norm{\hat{c}(s)}\;\right|ds =
\frac1\mu\int_{S^1}\left|\; \norm{c^\mu(s)}-1\right|{d\theta}\\
&\approx\frac1{2\mu}\sum_{i=1}^N(\atan(c^\mu_{i+1})-\atan(c^\mu_{i}))
\left(\left|\; \norm{c^\mu_{i+1}}-1\right|+\left|\; \norm{c^\mu_i}-1\right|
\right)
\end{align*}
where we exploit, that the boundary curve of the optimal solution is $\partial\hat{\Omega}=\mumat S^1$. \end{revision}
Furthermore, the tangential (and thus the unit normal) vector field for the iterates is computed by central differences.

The first idea for setting $\mu$ in the numerical experiments is $\mu=1$. However, then
$\hess{}f(\hat{\Omega})=2 \cdot id_{T_{\partial\hat{\Omega}}B_e}$ which means that the method of steepest descent with exact line search can exactly mimic the behaviour of the Newton method---including quadratic convergence. However, we want to experience the difference between steepest descent methods and Newton's method in this example. Therefore, we choose $\mu =2$ for performing the computations.
The initial shape is defined by its boundary $\partial\Omega^0:=c^0$ as
\[
c^0(s)=\frac12\left(
\begin{array}{l}
\cos(s)-0.15|1-\sin(2s)|\cos(s)\\
\sin(s)-0.15|1-\cos(2s)|\cos(s)
\end{array}
\right)
\, , \qquad s\in [0,2\pi]
\]
just to be somewhat more interesting than a simple circle. We use a piecewise linear approximation with 100 equidistant pieces in $[0,2\pi]$.
\begin{table}[h]
\caption{Performance of shape algorithms: steepest descent (indices SD) versus {\color{black}exact} Newton (indices NM), $\alpha$ denotes the line search parameter chosen by an exact line search applied to the objective.}
\begin{revision}
\begin{center}
\begin{tabular}{r|r|l|r||r|l|r}
It.-No.&$f(\Omega^k_{SD})$&$\bar{d}(\Omega^k_{SD},\hat{\Omega})$&$\alpha_{SD}$
&$f(\Omega^k_{NM})$&$\bar{d}(\Omega^k_{NM},\hat{\Omega})$&$\alpha_{NM}$\\\hline\hline
0&-0.5571&0.9222E+00& 0.50&-0.5571&0.9222E+00& 0.63\\
 1&-0.7630&0.2137E+00& 0.32&-0.7775&0.1382E+00& 0.98\\
 2&-0.7830&0.6174E-01& 0.36&-0.7854&0.3571E-02& 1.00\\
 3&-0.7852&0.1730E-01& 0.33&-0.7854&0.8187E-05& 1.00\\
 4&-0.7854&0.4888E-02& 0.34&-0.7854&0.1736E-09&\\
 5&-0.7854&0.1448E-02& 0.34\\
 6&-0.7854&0.4404E-03& 0.33\\
 7&-0.7854&0.1300E-03& 0.34\\
 8&-0.7854&0.4065E-04& 0.33\\
 9&-0.7854&0.1228E-04& 0.34\\
10&-0.7854&0.3909E-05& 0.33\\
11&-0.7854&0.1198E-05& 0.35\\
12&-0.7854&0.4026E-06& 0.32\\
13&-0.7854&0.1171E-06& 0.33\\
14&-0.7854&0.3650E-07& 0.25\\
15&-0.7854&0.8370E-08& 0.31\\
16&-0.7854&0.3041E-08&\\
\end{tabular}
\end{center}
\end{revision}
\label{tab-performance}
\end{table}
Table \ref{tab-performance} shows the performance of the shape steepest descent method versus the shape Newton method with exact line search, where we use $A=0$ in the definition of $G^A$ similarly as in \cite{Joshi07a}. 
\begin{revision}As in \cite{Joshi07a}, the potential pathologies of the case $A=0$ are numerically not felt, because the unavoidably finite discretization acts in a regularizing manner \end{revision}.The iterations are stopped at a distance of less than $10^{-7}$ to the solution. For the steepest descent method, we observe linear convergence with a factor of approximately $0.3$, while for the Newton method we see clear quadratic convergence. Figure \ref{fig-iterations} shows the various shape boundaries during the respective iterations.

\begin{figure}[h]
\label{fig-iterations}
\unitlength1cm
\begin{picture}(12,3)
\put(0,-4.8){\includegraphics[scale=0.30]{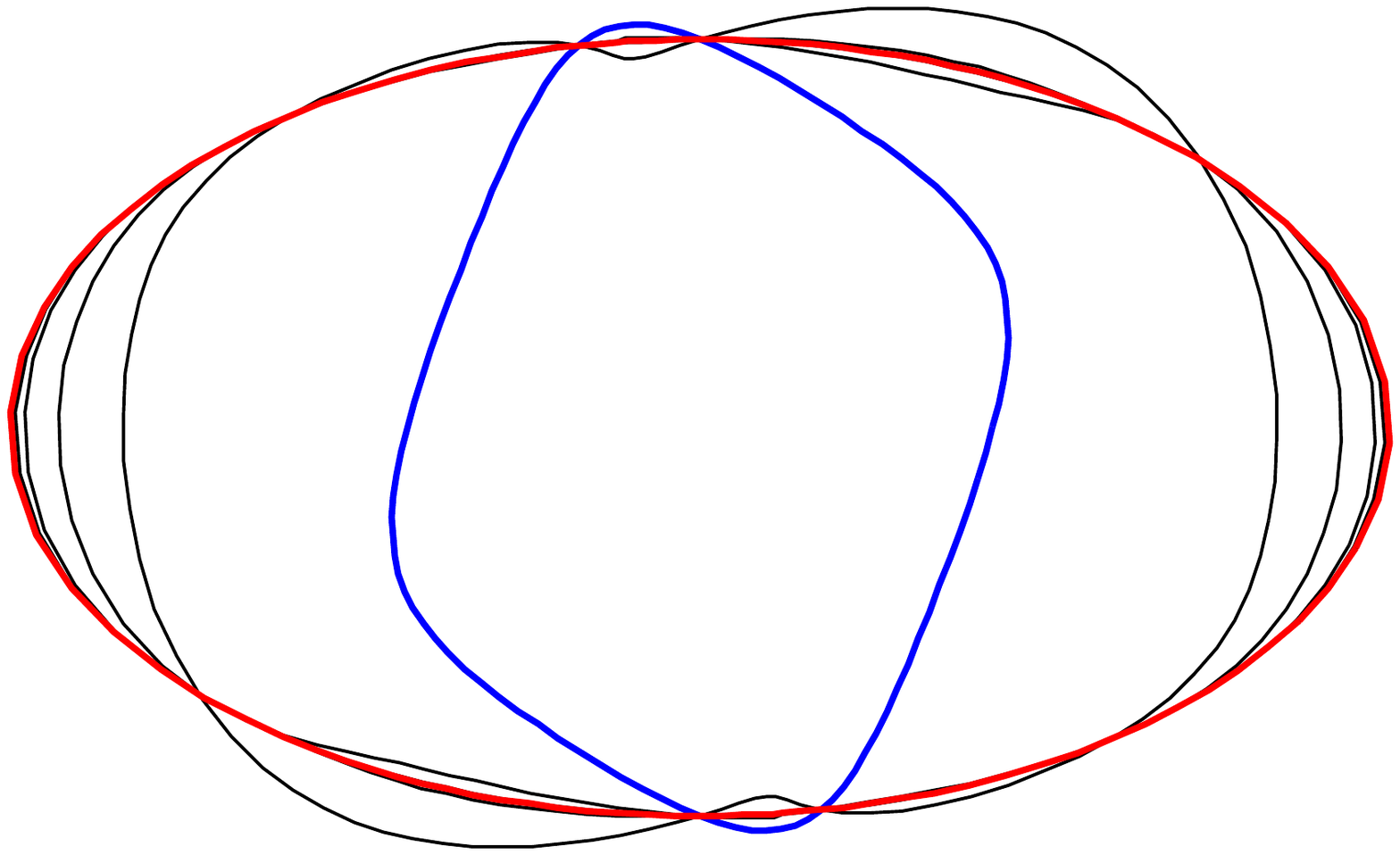}}
\put(6,-4.8){\includegraphics[scale=0.30]{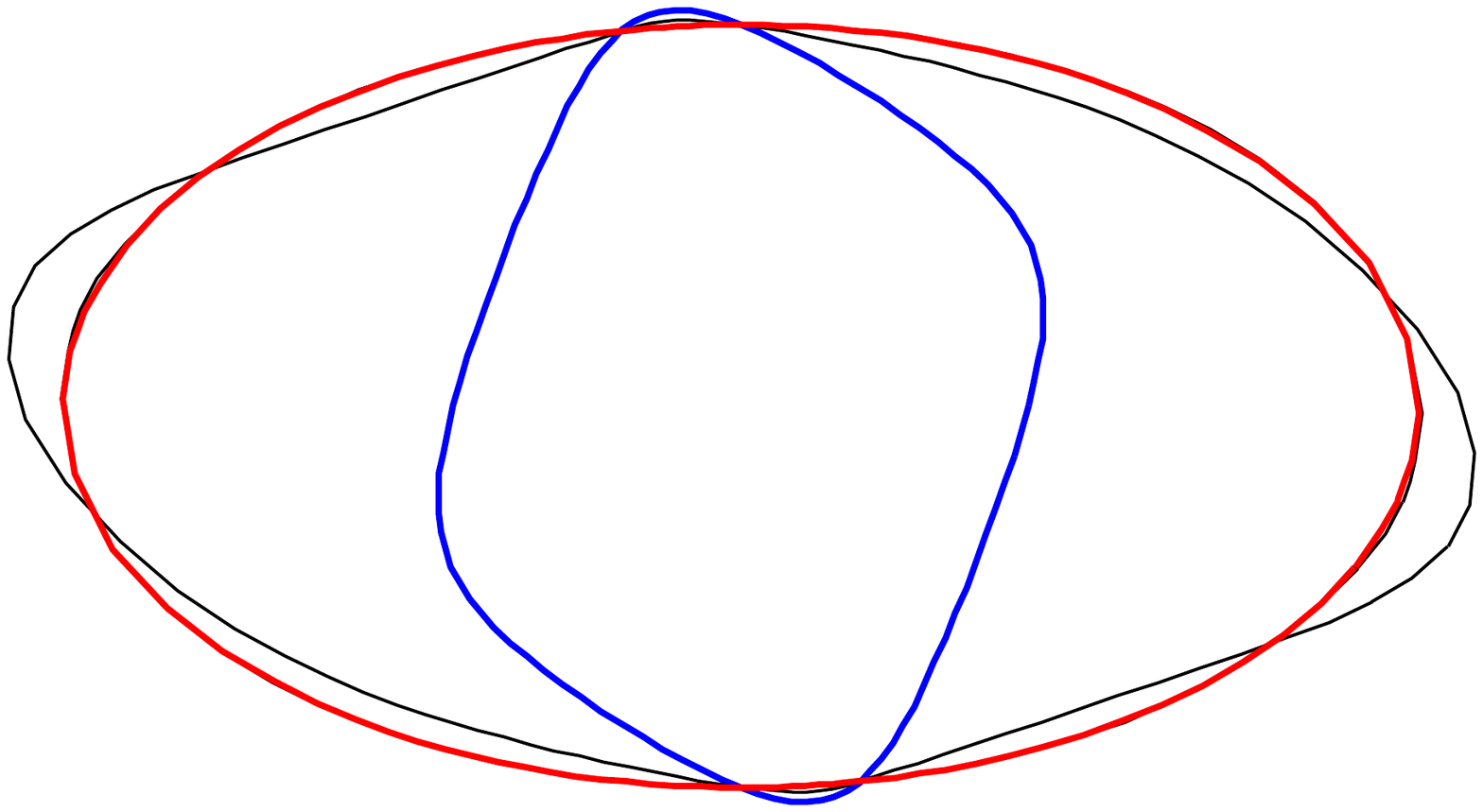}}
\end{picture}
\caption{Visualization of the shape iterates for the steepest descent (left) and for Newton's method (right) from initial (blue) to solution (red) shape.}
\end{figure}

\section{Conclusions}
In this paper, a novel point of view on shape optimization is presented---the Riemannian point of view. The Riemannian shape Hessian is introduced which can serve as a much more useful notion of second shape derivative than the classical so-called shape Hessian. With this term, classical optimization results, known in linear spaces, can be formulated and proved in the area of shape optimization. Newton optimization methods for shape optimization are analyzed. A new model problem for shape optimization is introduced which mimics the properties of a standard $L^2$-quadratic model problem in linear spaces---in so far as the Riemannian Hessian is just a multiplicative operator.  

Several aspects are now open for further research: This paper only deals with $C^{\infty}$-boundaries of 2D shapes. Of course, for practical purposes, this is not enough. The regularity has to be reduced, one has to go from dimension 2 to 3 and one has to treat shapes which are just part of submanifolds of $\R^3$, rather than  whole embeddings of the unit sphere. Furthermore, the implications for practical shape optimization have to be analyzed. This will all be covered in subsequent papers.

\section*{Acknowledgment} The author is very grateful for the hint of Oliver Roth (University of W\"urzburg, Germany) to the publication \cite{ShaMom2004}, where the whole endeavor of this paper began. \begin{revision}Furthermore, I would like to thank three anonymous referees, who have helped to polish the paper significantly.\end{revision}

%\bibliography{/Users/volkerschulz/VHS/research/publications/VHS-citations}{}
%\bibliographystyle{plain}

\end{document}